\documentclass[a4paper,11pt]{amsart}
\usepackage{amssymb,amsfonts,amsxtra,
mathrsfs,placeins,graphicx,verbatim,stmaryrd}
\usepackage[all]{xy}
\xyoption{line}
\usepackage{fullpage}
\usepackage{euscript}
\newtheorem{theorem}{Theorem}[section]
\newtheorem{cor}[theorem]{Corollary}

\newtheorem{prop}[theorem]{Proposition}

\theoremstyle{definition}

\newtheorem{example}[theorem]{Example}
\newtheorem{defi}[theorem]{Definition}
\newtheorem{rem}[theorem]{Remark}

\numberwithin{equation}{section}
\DeclareMathOperator{\Hom}{Hom}

\DeclareMathOperator{\End}{End}

\DeclareMathOperator*{\tw}{tw}

\newcommand{\noproof}{\begin{flushright} \ensuremath{\square}
\end{flushright}}
\def\ground{\mathbf{k}}
\def\C{\mathcal{C}}
\def\twC{\operatorname{tw}(\C)}

\def\id{\operatorname{id}}

\thanks{}

\begin{document}
\begin{abstract}
We introduce a certain differential graded bialgebra, neither commutative nor cocommutative, that governs perturbations of a differential on complexes supplied with an abstract Hodge decomposition. This leads to a conceptual treatment of the Homological Perturbation Lemma and its multiplicative version. As an application we give an explicit form of the decomposition theorem for $A_\infty$ algebras and $A_\infty$ modules and, more generally, for twisted objects in differential graded categories.
\end{abstract}
\title[On the perturbation algebra]{On the perturbation algebra}
\author{J. Chuang  \and A.~Lazarev}
\thanks{This work was partially supported by EPSRC grants EP/N015452/1 and EP/N016505/1}
\address{Centre for Mathematical Science\\City University\\London
EC1V 0HB\\UK}
\email{J.Chuang@city.ac.uk}
\address{University of Lancaster\\ Department of
Mathematics and Statistics\\Lancaster, LA1 4YF, UK.}
\email{a.lazarev@lancaster.ac.uk} 

\maketitle
\tableofcontents
\section{Introduction}
The Homological Perturbation Lemma (HPL) is a kind of a homotopy transfer theorem: it allows one to transfer a perturbed differential of a (co)chain complex onto another one, that is a strong deformation retract of it, see \cite{Cra} and references therein. One interesting feature of HPL is that it gives explicit universal formulae, as formal power series in the original strong deformation retract data. The standard proof of HPL is based on algebraic manipulations with certain linear operators and we introduce a certain algebra $A$ that captures the properties of these operators, leading to an abstract version of HPL, of which the ordinary HPL is a consequence. It should be noted that in \cite{BL} the term `Perturbation Algebra' is reserved for a different, though closely related, algebra, that is not differential graded (the differential is viewed as its element rather than an operator). It is possible that our results could be recovered in the framework of \cite{BL}, but we found that working with the smaller differential graded object is more natural.

The category of modules over $A$ is, essentially, the category of strong homotopy retractions (in fact, we find it more convenient to work with an equivalent notion of an abstract Hodge decomposition) with a perturbed differential. This category is non-symmetric monoidal, which corresponds to a certain comultiplication on ${A}$ turning it into a noncommutative and noncocommutative bialgebra. We introduce a certain localization $\hat{A}$ of $A$ which turns out also to be a bialgebra. This bialgebra, despite having a simple presentation by very few generators and relations, exhibits a surprisingly rich structure; it is a central object of study in this paper. We discover a certain remarkable endomorphism of  $\hat{A}$ that leads to a strengthened version of the HPL, as well as its multiplicative version~\cite{HK}. As a corollary, we obtain a decomposition theorem for $A_\infty$ algebras: every such algebra is isomorphic to the direct sum of a minimal $A_\infty$ algebra (having zero cohomology) and a linear contractible one. A similar result is also obtained for $A_\infty$ modules over an $A_\infty$ algebra.  The structure maps for the minimal module as well as for the $A_\infty$ isomorphism are explicitly described in terms of summation over certain trees. 

\subsection{Notation and conventions} We work in the category of $\mathbb Z$-graded (dg) vector spaces over a fixed ground field $\ground$; an object in this category is a pair $(V,d_V)$ where $V$ is a graded vector space and $d_V$ is a differential on it; it will always be assumed to be of cohomological type (so it raises the degree of a homogeneous element). Unmarked tensor products will be understood to be taken over $\ground$; we will abbreviate `differential graded' to `dg'. The \emph{suspension} of a graded vector space $V$ is the graded vector space $\Sigma V$ so that $(\Sigma V)^i=V^{i+1}$.

The graded $\ground$-linear dual to a graded vector space $V$ is denoted by $V^*$; it is a \emph{pseudo-compact} vector space, i.e. an inverse limit of finite-dimensional (graded) vector spaces. Pseudo-compact vector spaces are endowed with the topology of the inverse limit and form a category where maps are required to be continuous. The category of pseudo-compact vector spaces is anti-equivalent to that of (discrete) vector spaces via linear duality; we refer to \cite{VdB} for a short account on pseudo-compact spaces. All results and constructions of this paper are valid when we work in the underlying category of pseudo-compact graded vector spaces.

A dg algebra is an associative monoid in the dg category of dg vector spaces with respect to the standard monoidal structure given by the tensor product. A dg vector space $V$ is a (left) dg module over a dg algebra $A$ if it is supplied with a dg map $A\otimes V\to V$ satisfying the usual conditions of associativity and unitality.

A Maurer-Cartan (MC) element in a dg algebra $A$ is an element $x$ of cohomological degree 1 satisfying the MC equation $d(x)+x^2=0$. It determines  a twisted differential in $A$ by the formula $d^x(a):=d(a)+[x,a]$ for  $a\in A$; the MC equation ensures that $(d^x)^2=0$. We will denote the dg algebra $A$ supplied with the differential $d^x$ by $A^x$. The MC element $x$ also determines a twisting of any dg $A$-module $(M,d_M)$; namely the differential $d+x$ (where $x$ is viewed as an operator on $M$) makes $M$ into a left dg $A^x$-module as could easily be checked; we will denote this dg $A$-mod by $M^{[x]}$.  The group $A_0^{\times}$ of invertible elements $\gamma$ of $A_0$ acts on the set of MC elements in $A$ by via the gauge action:
$\gamma\cdot x : =\gamma x \gamma^{-1}-d(\gamma)\gamma^{-1}.$

The paper is organized as follows. In Section \ref{abstractHD} we recall the definition of an abstract Hodge decomposition and introduce a dg bialgebra governing this structure. In  Section \ref{main} the Perturbation Algebra, our main object of study, is introduced and studied in some detail; this is, in fact a dg bialgebra, neither commutative nor cocommutative. We formulate and prove an abstract version of the HPL, of which the ordinary HPL is a corollary. We further strengthen HPL by proving that the perturbed differential is actually conjugate to a transferred one in a suitable sense. A decomposition theorem for twisted objects in a differential graded category is a consequence of the strengthened HPL. Specializing further, we obtain an explicit form of the decomposition theorem for $A_\infty$ modules.
\section{Category of abstract Hodge decompositions}\label{abstractHD}
In this section we will recall the definition of an abstract Hodge decomposition  on a dg vector space; this terminology was introduced in \cite{CL}.
\begin{defi}\label{def_H}
An abstract Hodge decomposition on a dg vector space $(V,d)$ is a pair of operators $t,s$  on $V$
$$t:V\to V \quad \text{and} \quad s:V\to V,$$
of degrees $0$ and $-1$ respectively,
 such that
\begin{enumerate}
\item $s^2=0$,
\item $sd+ds=1-t$,
\item $dt=td$,
\item $t^2=t$,
\item $st=ts=0$,
\end{enumerate}
We will also refer to the quadruple $(V,d,s,t)$ (or the triple $(V,s,t)$ if $d$ is clear from the context) as an abstract Hodge decomposition and use the shorthand HD for it.
\end{defi}
\begin{rem}
The notion of an HD could be reformulated as follows.  Let $U:=\operatorname{Im}(t)$ be the image of the operator $t$. It follows from the identity (3) above that the differential on $V$ restricts to a differential on $U$ and that the projection $t:V\to U$ and the inclusion $i:U\hookrightarrow V$ determine chain maps between the dg spaces $U$ and $V$. The conditions (2) and (4) above imply that $t\circ i=\id_U$ and that $i\circ t$ is chain homotopic to the identity on $V$. Thus, the pair of dg vector spaces $V$ and $U$ constitute a strong deformation retraction data. The condition $s^2=0$, the so-called \emph{side condition}, can always be imposed at no cost. It is more convenient for us to view the above as the structure on a single space $V$.
\end{rem}
Let $(V, t_V, s_V)$ and $(W,t_W,s_W)$ be dg vector spaces with the choice of an HD. Then the operators $t\otimes t$ and $s\otimes 1+t\otimes s$ determine an HD on the vector space $V\otimes W$. A straightforward inspection shows that this turns the category of HDs into a (non-symmetric) monoidal category. We will denote this category by $\mathcal H$.

The content of an HD is captured by a certain dg bialgebra. The following result is obvious.
\begin{prop}
Let $H$ be an associative algebra spanned by three vectors $s,t$ and $1$, with the relations $st=ts=s^2=0$ and $t^2=t$ and the differential $d(s)=1-t$, $d(t)=0$. There is a coassociative coproduct $\Delta:H\to H\otimes H$ so that $\Delta(t)=t\otimes t$ and $\Delta(s)=s\otimes 1+t\otimes s$ and a counit $\epsilon:H\to \ground: \epsilon(t)=1, \epsilon(s)=0$; these turn $H$ into a bialgebra. Furthermore, the monoidal categories of $dg$ $H$-modules and of HDs are isomorphic.
\end{prop}
\noproof
\begin{rem}\label{Bialgebra_H}\
\begin{enumerate}
\item
The bialgebra $H$ lacks an antipode since the grouplike element $t$ is a zero divisor, hence non-invertible.
\item
The dg abelian group $H$ is isomorphic to the normalized simplicial chain complex of the unit interval $[0,1]$ with two 0-simplices and one 1-simplex; morever, the coalgebra structure corresponds to Alexander-Whitney map; this was observed in \cite{SS}, pp 503-504. The commutative product on $H$ corresponds to the ordinary multiplication on $[0,1]$.
\end{enumerate}
\end{rem}
\begin{defi}\label{Hodge_alg}
We will call a monoid in $\mathcal H$ a \emph{multiplicative} HD.
\end{defi}
Recall that given two endomorphisms $f,g\in\End(A)$ of an associative algebra $A$, a linear map $s:A\to A$ is called an $(f,g)$-derivation, if for any $a,b\in A$ we have the `$(f,g)$-Leibniz rule':  $s(ab)=s(a)g(b)\pm f(a)s(b)$ where the sign is determined by the standard Koszul rule. Note that $f-g$ is always  both an $(f,g)$-derivation and a $(g,f)$-derivation.  We can give the following characterization of multiplicative HDs:
\begin{prop}
A dg algebra $A$ is a multiplicative HD if its underlying dg vector space is supplied with operators $s$ and $t$ satisfying the conditions of Definition \ref{def_H}, and such that $t$ is an algebra endomorphism of $A$ and $s$ is a $(t,\id)$-derivation.
\end{prop}
\begin{proof}
The condition that $A$ is a multiplicative HD is equivalent to the multiplication map $A\otimes A\to A$ being an $H$-module map. Taking into account the comultiplication in $H$ we see that the latter is equivalent to $t$ being an algebra map and $s$ being a $(t,\id)$-derivation.
\end{proof}
\begin{example}\label{tensor}
Let $(V, s, t)$ be an HD. Then the tensor algebra $T(V):=\bigoplus_{n=0}^\infty V^{\otimes n}$ is a multiplicative HD where $t$ is extended to an algebra endomorphism $t_{T(V)}$ of $T(V)$ component-wise, $s_{T(V)}$ is extended from $s$ by the $(t_{T(V)},\id)$-Leibniz rule and $d_{T(V)}$ is extended from $d$ by the ordinary Leibniz rule. Indeed, the identities \[s^2_{T(V)}=t_{T(V)}s_{T(V)}=s_{T(V)}t_{T(V)}=0 \quad \text{and} \quad t^2_{T(V)}=t_{T(V)}\] are immediate. To see that $d_{T(V)}(s_{T(V)})=\id-t_{T(V)},$ note that both sides of the last identity are $(t_{T(V)},\id)$-derivations and so it is sufficient to check on $V$ where it is clear.
\end{example}
\begin{rem}\label{pseudo}
We considered the notion of an abstract Hodge decomposition in the category of dg vector spaces; one can define it, completely analogously, in the category of pseudo-compact dg vector spaces. The category of pseudo-compact dg vector spaces is symmetric monoidal, and a monoid in this category is, essentially, the same as a dg coalgebra by linear duality. There is, then, the notion of a pseudo-compact multiplicative HD analogous to Definition \ref{Hodge_alg}. We will skip the (rather obvious) details.
\end{rem}
\section{The perturbation algebra}\label{main}
\subsection{Definition and first properties} We will now introduce our main object of study: the perturbation dg algebra $A$ as well as its localization $\hat{A}$.
\begin{defi}
The perturbation algebra $A$ is the associative dg bialgebra generated by the elements $s,t$ and $x$ of degrees $-1,0$ and $1$ respectively, subject to the relations
\begin{align*}s^2&=st=ts=0\\ t^2&=t.
\end{align*}
The differential on $A$ is specified by
\begin{align*}&d(s)=1-t;\quad d(t)=0\\ &d(x)=-x^2.
\end{align*}
The comultiplication on $A$ is specified by
\begin{align*}&\Delta(s)=s\otimes 1+t\otimes s;\\
&\Delta(t)=t\otimes t;\\
&\Delta(x)=x\otimes 1+1\otimes x.\end{align*}
Additionally, the bialgebra $A$ possesses a counit:
\begin{align*}&\epsilon(s)=\epsilon(x)=0;\\ &\epsilon(t)=1.\end{align*}
\end{defi}
Thus, $A$ is obtained from $H$ by freely adjoining the primitive element $x$ satisfying the MC equation $dx=-x^2$. Just as $H$, $A$ does not possess an antipode. Nevertheless, there exists a (unique, involutory) dg algebra anti-automorphism $\rho:A\to A$ fixing $s$ and $t$, and sending $x$ to $-x$; using $\rho$ one can define an $A$-module structure on a $\ground$-linear dual to an $A$-module. Rather than being a coalgebra anti-automorphism like a Hopf algebra antipode, $\rho$ is a coalgebra automorphism; consequently, given $A$-modules $V$ and $W$, the standard isomorphism of dg vector spaces
$(V\otimes W)^*\cong V^*\otimes W^*$ is an isomorphism of $A$-modules.

The algebra $A$ governs `perturbations' of the differential on a given dg vector space supplied with the choice of an HD. More precisely, we have the following result.
\begin{prop}
Let $(V, d_V, s,t)$ be an HD and $x$ be an endomorphism of $V$ of degree 1 and such that $(d_V+x)^2=0$. Then $V$ is a dg module over $A$; conversely any dg $A$-module is of the kind just described.
\end{prop}
\begin{proof}
We just need to note that the equation $(d_V+x)^2=0$ is equivalent to $[d_V,x]+x^2=0$. Thus, viewing the operators $x,s$ and $t$ as elements in $\End(V)$ determines a map of dg algebras $A\to \End(V)$ where $\End(V)$ is supplied with the differential $[?,d_V]$. This map is equivalent to giving $V$ the structure of an $A$-module.
\end{proof}
We also have a multiplicative version of this result; the proof is obvious.
\begin{prop}
Let $(V,d_V s,t)$ be a multiplicative HD and $x$ be a derivation of $V$ of degree 1 and such that $(d_V+x)^2=0$. Then $V$ is a dg algebra over $A$ (i.e. the multiplication map $V\otimes V\to V$ is a dg $A$-module map); conversely any dg $A$-algebra is of the kind just described.
\end{prop}
\noproof
We will also be interested in the dg algebra $\hat{A}$ obtained from $A$ by formally inverting the elements $1+sx$ and $1+xs$. In fact, invertibility of one of these elements implies that the other one is also invertible, e.g. $(1+xs)^{-1}=x(1+sx)^{-1}s+1$ as was observed in \cite[Remark 2.4]{Berg}.  Then $\hat{A}$ turns out to be a dg bialgebra (despite $1+sx$ and $1+xs$ not being grouplike). The dg bialgebra $\hat{A}$ could be realized as a dg subbialgebra of the completion of $A$ at the two-sided ideal generated by $x$. The inverses of $1+sx$ and $a+xs$ are represented by power series $\sum_{n=0}^{\infty}(-1)^n(sx)^n$ and $\sum_{n=0}^{\infty}(-1)^n(xs)^n$ respectively. We will denote these inverses by $\alpha:=(1+sx)^{-1}$ and $\beta:=(1+xs)^{-1}$.

\subsection{Abstract Perturbation Lemma and its consequences} We will now delve a little deeper in the structure of the algebra $\hat{A}$, and, in particular, prove an abstract version of the Homological Perturbation Lemma.

Note the following useful identities, to be used repeatedly in calculations:
\begin{equation}\label{easy}\begin{split}
\alpha s &=s\beta, \quad
x\alpha = \beta x,\\
s\alpha &=s, \quad
\beta s = s, \\
 t\alpha &= t, \quad \beta t= t,\\
\beta\alpha &=\alpha + \beta -1.
 \end{split}
\end{equation}
These are all obvious except, perhaps, the last one. To deduce it, note that $\alpha-1=-sx\alpha$ and
 $\beta-1=-\beta xs$ from which it follows that $(\beta-1)(\alpha-1)=0$ or
 $\beta\alpha= \alpha + \beta -1$.
Furthermore, the following identities hold:
\begin{equation}\label{easy'}\begin{split}
d(\alpha) = (\alpha t - 1)x\alpha, \\ d(\beta)= \beta x (1-t \beta).\end{split}
\end{equation}
The first formula above is given by the calculation
\[d(\alpha)=-\alpha d(\alpha^{-1}) \alpha =	-\alpha ((1-t)x-s(-x^2))\alpha = \alpha  (t-\alpha^{-1}) x\alpha,\]
 and the second formula by a similar calculation, or by using the anti-involution $\rho$.

\begin{prop}\label{phi} There exists a unique algebra automorphism
	$$\phi:\hat{A} \to \hat{A}$$
	such that
\begin{align*}\phi(s)&= \alpha s =s\beta, \\ \ \phi(t)&=\alpha t \beta, \quad \phi(x)=-x, \\
 \phi(\alpha)&=\alpha^{-1}, \quad \phi(\beta)=\beta^{-1}
\end{align*}
	It is an involution commuting with $\rho$. Moreover, there is a dg bialgebra structure on $\hat{A}$ extending that on $A$; with this structure $\phi$ is a bialgebra automorphism.
	\end{prop}
\begin{proof}
	If such an endomorphism taking the stated values exists, it is clearly an involutory automorphism that commutes with $\rho$. To check that such a $\phi$ is a well-defined algebra endomorphism,  we compute, using (\ref{easy})
\begin{align*}(\alpha s)(\alpha s)&=(\alpha s)(s\beta)=0, \\
(\alpha t \beta)^2 
&=\alpha t(\alpha+\beta-1)t\beta
=\alpha t \beta +\alpha t \beta-\alpha t \beta\\
&=\alpha t \beta,\\ (\alpha s) (\alpha t \beta)&=\alpha st\beta=0, \\
(\alpha t \beta) (s\beta)&=\alpha ts\beta=0,\\
\alpha^{-1}\phi(\alpha^{-1})&=\alpha^{-1}[1+(\alpha s)(-x)]\\
&=\alpha^{-1}-sx=1,\\
\phi(\beta^{-1})\beta^{-1}&=[1+(-x)(s\beta)]\beta^{-1}\\
&=\beta^{-1}-xs= 1.
\end{align*}
To deal with the comultiplication, let us embed $\hat{A}$ into $\tilde{A}$, the  completion of $A$ by the two-sided ideal generated by $x$. Note that $\tilde{A}$ is a (topological) dg bialgebra; it is topologically generated by $s,t$ and $x$ and the diagonal $\Delta:\tilde{A}\to \tilde{A}\hat{\otimes}\tilde{A}$ is continuous. Let us check compatibility of $\phi$ and $\Delta$ in $\tilde{A}$; it suffices to check it on $s,t$ and $x$.

Since $x$ and $\phi(x)=-x$ are both primitive, we just need to show that $\phi(t)=\alpha t\beta$ is group-like and $\phi(s)=\alpha s$ is skew-primitive with respect to the idempotent $\alpha t\beta$, i.e.
\begin{equation}\label{skew}\Delta(\alpha s)=\alpha s\otimes 1+\alpha t\beta\otimes \alpha s.
\end{equation} We have:
\begin{align*}
\Delta(\alpha^{-1})(\alpha t\otimes \alpha t)&=\Delta(1+sx)(\alpha t\otimes \alpha t)\\
&=[(1+sx)\otimes 1+s\otimes x-tx\otimes s+t\otimes sx](\alpha t\otimes \alpha t)\\
&=t\otimes\alpha t+t\otimes sx\alpha t\\
&=t\otimes(1+sx)\alpha t\\
&=t\otimes t.
\end{align*}
This shows that $\alpha t$ is grouplike. Similarly $t\beta$ is grouplike and then $\alpha t\beta=(\alpha t)(t\beta)$ is likewise grouplike. Next,
\begin{align*}
\Delta(\alpha^{-1})(\alpha s\otimes 1+\alpha t\beta\otimes\alpha s)&=[(1+sx)\otimes 1+s\otimes x-tx\otimes s+t\otimes sx](\alpha s\otimes 1+\alpha t\beta\otimes\alpha s)\\
&=s\otimes 1+t\beta\otimes\alpha s+tx\alpha s\otimes s+t\beta\otimes sx\alpha s\\
&=s\otimes 1+t\beta\otimes s+tx\alpha s\otimes s\\
&=s\otimes 1+t\otimes s.
\end{align*}
This implies (\ref{skew}) and shows that $\phi$ is a coalgebra endomorphism of $\tilde{A}$. Thus, $\phi\hat{\otimes}\phi\circ \Delta=\Delta\circ \phi$ and $\Delta=\phi\hat{\otimes}\phi\circ\Delta\circ \phi^{-1}$ on $\tilde{A}$. This allows one to restrict $\Delta$ on $\hat{A}$ by $\Delta({\alpha}):=(\phi\otimes\phi)\Delta(\alpha^{-1})$ and $\Delta({\beta}):=(\phi\otimes\phi)\Delta(\beta^{-1})$.
\end{proof}	
The endomorphism $\phi$ is, however, not a dg endomorphism of $\hat{A}$. Consider $\hat{A}^x$, the algebra $\hat{A}$ supplied with the twisted differential
$d^x=d+[x,?]$. The following result is an abstract version of the HPL.
\begin{theorem}\label{abstract_lemma}\
\begin{enumerate}
\item
We have the following equation of endomorphisms of $\hat{A}$:
\begin{equation}\label{com}\phi d = d^x \phi.\end{equation}
Thus $\phi$ defines isomorphisms of dg bialgebras:
\[\phi: \hat{A} \to \hat{A}^x,\]\[ \phi: \hat{A}^x\to \hat{A}.\]
\item
The map $tat\mapsto\alpha (tat)\beta$ determines an isomorphism of dg bialgebras
\[
(t\hat{A}t,d+tx\alpha t)\xrightarrow[]{\cong} ((\alpha t\beta )\hat{A}(\alpha t\beta), d^x)
\]
\end{enumerate}
\end{theorem}
\begin{proof}
Both $\phi d \phi$ and $d^x$ are derivations of $\hat{A}$, so it suffices to check (\ref{com}) on the generators $s$, $t$ and $x$.
We have
\begin{align*}\phi (d(x))&=\phi(-x^2)\\&= -(\phi(x))^2\\&= -x^2,
\end{align*} and similarly,
\begin{align*}d^x (\phi(x))&= -(-(-x)^2)+[x,-x]\\&=-x^2
\end{align*}
Using (\ref{easy}) and (\ref{easy'}) we obtain:
\begin{align*}
	d^x(\phi(s)) & = d^x (\alpha s) \\
	 & = (\alpha t  - 1)x\alpha s + \alpha (1-t) + x\alpha s + \alpha sx \\
	 & = \alpha t \beta xs + \alpha (1+sx-t) \\
	 & = 1-\alpha t(1-\beta xs) \\
	 & = 1-\alpha t \beta \\
	 & = \phi(d(s)).
 \end{align*}
Finally,
\begin{align*}d^x(\phi(t))&= d^x (1-\phi(d(s)))\\
&=d^x(1-d^x(\phi(s)))\\&=0=\phi(d(t)).
\end{align*}
This finishes the proof of the first claim. To prove the second, we note first that, since $t$ and $\alpha t\beta$ are grouplike idempotents, the spaces $t\hat{A}t$ and $\alpha t\beta\hat{A}\alpha t\beta$ are (graded) bialgebras. To see that they are multiplicatively isomorphic,  let $a,b\in\hat{A}$, then we have:
\begin{align*}
(\alpha tat\beta)(\alpha tat\beta)=&\alpha ta(t\beta\alpha t)at\beta\\
=&(\alpha ta)(tb\beta)
\end{align*}
as required. Next, since  $d^x(\alpha t\beta)=0$, the vector space and $\alpha t\beta \hat{A}\alpha t\beta$
is closed with respect to $d^x$ respectively and thus, is a dg algebra. Similarly, $t\hat{A}t$ is closed with
 respect to $d+tx\alpha t$.  The following calculation shows that the actions of $d^x$ and $d+tx\alpha t$
 are compatible (and thus, $d+tx\alpha t$ squares to zero in $t\hat{A}t$). Let $a\in \hat{A}$. We have:
\begin{align*}
d^x(\alpha tat\beta)&=d(\alpha tat\beta)+[x,\alpha tat\beta]\\
&=(\alpha t-1)x\alpha tat\beta+(-1)^{|a|}\alpha tat\beta x(1-t\beta)+\alpha td(a)t\beta+x\alpha tat\beta-(-1)^{|a|}\alpha tat\beta x\\
&=\alpha t(da+x\alpha ta-(-1)^{|a|}at\beta x)t\beta\\
&=\alpha t(da+[a,tx\alpha t])t\beta
\end{align*}
as claimed.
\end{proof}
\begin{cor}\label{maurer}\
\begin{enumerate}
\item
The element $\xi:=tx\alpha t(=t\beta x t)\in \hat{A}$ is an MC element in $\hat{A}$.
\item It is $t$-primitive, i.e. $\Delta(\xi)=\xi\otimes t+t\otimes \xi$.
\end{enumerate}
\end{cor}
\begin{proof}
We saw that the twisted differential $d^{\xi}:a\mapsto d_A(a)+[\xi a]$ squares to zero in $t\hat{A}t$ and it follows that $\xi$ is an MC element in the dg algebra $t\hat{A}t$. Since $t\hat{A}t$ is a retract of $\hat{A}$, the element $\xi$ is also MC in $\hat{A}$. Furthermore,
\begin{align*}
\Delta(\xi)&=\Delta(tx\alpha t)\\
&=(t\otimes t)(x\otimes 1+1\otimes x)(\alpha t\otimes \alpha t)\\
&=tx\alpha t\otimes t+tx\alpha t\otimes t\\
&=\xi\otimes t+t\otimes \xi.
\end{align*}
\end{proof}
The ordinary HPL is a consequence of the abstract one:
\begin{cor}\label{ordinary}\
\begin{enumerate}
\item Let $(V, d_V,s,t)$ be an HD and $x\in\End(V)$ is such that $[d,x]+x^2=0$. Additionally, we assume that the operators $1+sx$ and $1+xs$ act invertibly on $V$. Then
$(V,d_V+x,\alpha s,\alpha t\beta)$ is also an HD. If $(V,d_V,s,t)$ is a multiplicative HD and $x$ was a
derivation, then $(V,d_V+x,\alpha s,\alpha t\beta)$ is also a multiplicative HD.
\item
The operator $\alpha t\in \End(V)$ restricts to an isomorphism of dg $t\hat{A}t$-modules, which is multiplicative if the HD $(V,d_V,s,t)$ is multiplicative:
\[\alpha t: (tV, d_V+tx\alpha t)\to(\alpha t\beta V, d_V+x),\] where $t\hat{A}t$ acts on $(\alpha t\beta V, d_V+x)$ via the dg algebra map $t\hat{A}t\to (\alpha t\beta)\hat{A}(\alpha t\beta)$ constructed in Theorem \ref{abstract_lemma}. The inverse isomorphism is given by $t\beta\in \End(V)$:
\[t\beta: (\alpha t\beta V, d_V+x)\to(tV, d_V+tx\alpha t).\]
\end{enumerate}
\end{cor}
\begin{proof}
The vector space $V$ supplied with the perturbed differential $d+x$ is a dg module over the twisted perturbation algebra
$(A,d^x)$. Therefore the composite map
\[\xymatrix{H\ar@{^{(}->}[r]& (\hat{A}, d)\ar^{\phi}[r]&(\hat{A},d^x)}\]
makes $V$ into an $H$-module (an $H$-algebra in the multiplicative case) with elements $t, s\in H$ acting on $V$ as operators $\alpha t\beta$ and $s\alpha$ respectively. This proves statement (1). For statement (2) note first that since $d(t)=0$ in $\hat{A}$, the differential $d_V$ restricts to $tV\hookrightarrow V$, the image of the projector $t$, making it a dg vector space. Similarly, $\alpha t\beta V$ is a dg vector space with respect to the perturbed differential $d_V+x$. In the multiplicative case both $tV$ and $\alpha t\beta V$ are dg algebras.  Statement (2) above now follows form the corresponding statement of Theorem \ref{abstract_lemma}. The claim about the inverse isomorphism also follows, taking into account the identity $t\alpha\beta t=t$.
\end{proof}
\begin{rem}
The HPL is particularly useful when $(V,s,t)$ is a \emph{harmonious} HD, i.e. when $tV$ carries the homology of $V$. In that case the differential $d_V$ vanishes in $tV$ and we have an isomorphism of dg vector spaces
\[
(tV,tx\alpha t)\cong (\alpha t\beta d_V, V+x).
\]
\end{rem}
\subsection{Minimal models for $A_\infty$ algebras} We will now outline a standard application of the HPL, cf. for example \cite{Hue}, \cite{markl}.

Recall that an $A_\infty$ algebra structure on a dg vector space $(V,d_V)$ is a continuous derivation $m=m_2+m_3+\ldots$ of the completed tensor algebra $\hat{T}\Sigma V^*$, the free local pseudo-compact algebra on the pseudo-compact vector space $V^*$; here $m_i$ is the component of $m$ that raises the tensor degree by $i-1$. The derivation $m$ is determined by its restriction on $\Sigma V^*$ that we will also denote by $m$. Thus, $m:\Sigma V^*\to \hat{T}\Sigma V^*$ and $m_n:\Sigma V^* \to (\Sigma V^*)^{\hat{\otimes} n}$. Note that this definition is equivalent by dualization to the, perhaps, more familiar one given in terms of multilinear maps $V^{\otimes n}\to V$.

Furthermore, an $A_\infty$ map between two $A_\infty$ algebras $(V,d_V,m_V)$ and $(U,d_U,m_U)$ is a continuous algebra map $f=f_1+f_2+\ldots :\hat{T}\Sigma U^*\to \hat{T}\Sigma V^*$ such that $f\circ (d_U+m_U)=(d_V+m_V)\circ f$; here $f_n$ is the component of $f$ raising the tensor degree by $n-1$. The map $f$ it is called a \emph{weak equivalence} if $f_1:\Sigma U^*\to\Sigma V^*$ is a quasi-isomorphism; there is also an appropriate notion of a homotopy between two $A_\infty$ maps. We refer to \cite{Ke, HL} for details on $A_\infty$ algebras.

If $V$ is supplied with an HD then so is $V^*$ and therefore $\hat{T}V^*$, cf. Example \ref{tensor}; there the definition was given for an ordinary (non-completed) tensor algebra, but the completed case is similar, cf. Remark \ref{pseudo}. Additionally, the action of $A$ on $\hat{T}V^*$ extends to an action of $\hat{A}$ (for this the completeness of the tensor algebra is necessary).

If the underlying space of an $A_\infty$ algebra, has vanishing differential, the corresponding $A_\infty$ algebra is called \emph{minimal}. The well-known result of Kadeishvili \cite{Kad} states that there for any $A_\infty$ algebra $(V,d_V, m_V)$ is an $A_\infty$ structure on $H^*(V)$ that is weakly equivalent to $(V,d_V,m_V)$. More recently, many proofs of this important results appeared, particularly those giving explicit formulas for minimal models. Perhaps the most straightforward and compact way to obtain these formulas is through the HPL. The following result was essentially formulated in \cite{Hue}. Its proof is just the unwrapping of the (pseudo-compact analogue of) Corollary \ref{ordinary}.
\begin{theorem}\label{minimal_model}
Let $(V, d, m)$ be an $A_\infty$ algebra and suppose that $V$ is supplied with operators $s,t$ constituting an HD on $V$, we will denote by $\tilde{t},\tilde{s}$ the corresponding operators giving a multiplicative HD on $\hat{T}\Sigma V^*$ (cf. Example \ref{tensor} for the construction in the discrete case). Then:
\begin{enumerate}
\item
The dg space $tV$ supports the structure of an $A_\infty$ algebra given by \[m_{t(V)}:=\tilde{t}m(1+\tilde{s}m)^{-1}\tilde{t};\]
 \item
 There are $A_\infty$ maps between $(V,m)$ and $tV,m_{t(V)}$ represented by
 \[
 (1+\tilde{s}m)^{-1}\tilde{t}:\hat{T}\Sigma (tV)^*\to  \hat{T}\Sigma V^*;
 \]
 \[
 \tilde{t}(1+m\tilde{s})^{-1} :\hat{T}\Sigma V^*\to \hat{T}\Sigma (tV)^*,
 \]
 such that
 \begin{align*}\tilde{t}(1+m\tilde{s})^{-1} )\circ(1+\tilde{s}m)^{-1}\tilde{t}&=\id_{tV} \quad \text{and}\\
  \tilde{t}(1+\tilde{s}m)^{-1} )\circ(1+m\tilde{s})^{-1}\tilde{t}&=\id-\tilde{s}(1+\tilde{s}m)^{-1}.
  \end{align*}
\end{enumerate}
\end{theorem}\noproof
This is a compact version of the minimal model theorem for $A_\infty$ algebras; writing the corresponding formulas in components leads to summations over certain decorated planar trees as explained in the papers of Markl and Huebschmann cited above.
\subsection{Gauge equivalence and the decomposition theorem}
Let $V$ be a dg $\hat{A}$ module, in other words a dg vector space with an HD and an MC operator $x\in \End(V): [d,x]=-x^2$, such that $1+sx$ and
$1+xs$ act invertibly on $V$. We saw in Corollary \ref{ordinary} that dg spaces $(tV, d_V+t\alpha xt)$ and $(\alpha t\beta V, d_V+x)$ are isomorphic. Note that since $t\alpha xt$ is an MC element (either in $\End(V)$ or $\hat{A}$, see Corollary \ref{maurer}), the operator $d_V+t\alpha xt$ squares to zero on the whole of $V$. It makes sense to ask, therefore, whether the operators $d_V+t\alpha xt$ and $d_V+x$ are conjugate or, equivalently, that the MC elements $x$ and $t\alpha xt$ are gauge equivalent  in $\End(V)$. This turns our to be true (see Corollary \ref{conj} below), in fact we will show that these elements are gauge equivalent in $\hat{A}$ (Theorem \ref{decomposition} below).
First we will introduce a certain invertible element in $\hat{A}$ closely related to the idempotent $\alpha t\beta$.
\begin{prop}\label{more_identities}
	Let $g:= (1+t-\alpha t)\beta.$ Then $g$ is invertible and  \[g^{-1}=\phi(g)=\beta^{-1}(1-t+\alpha t).\]Moreover, the following identities holds
 \begin{align*}
 (\alpha t\beta) g(\alpha t\beta)&=\alpha t\beta\\tg&=t\beta\\ g^{-1}t&=\alpha t.\end{align*}
	\end{prop}
\begin{proof}
Using Proposition \ref{phi} we compute:
\begin{align*}
\phi(g)&=\phi(1+t-\alpha t)\phi(\beta)\\
&=(1+\alpha t\beta-t\beta)\beta^{-1}\\
&=\beta^{-1}(1-t+\alpha t).
\end{align*}
 Furthermore, taking into account the identity $t\alpha=t$ we have
\begin{align*}
(1-t+\alpha t)(1+t-\alpha t)&=1+t-\alpha t-t-t-t\alpha t+\alpha t+\alpha t-\alpha t\alpha t\\
&=1
\end{align*}
 Then also $\phi(g) g = \phi(g\phi(g))=\phi(1)=1$. To prove the stated identities we compute :
\begin{align*}
\alpha t\beta g\alpha t\beta&=\alpha t\beta (1+t-\alpha t)\beta\alpha t\beta\\
&=(\alpha t\beta+\alpha t-\alpha t(\alpha+\beta-1)t)\beta\alpha t\beta\\
&=\alpha t\alpha t\beta\alpha t\beta\\
&=\alpha t\beta\alpha t\beta\\
&=\alpha t\beta.
\end{align*}
Next,
\begin{align*}
tg=&t(1+t-\alpha t)\beta t\\=&(2t-t\alpha t)\beta=t\beta.
\end{align*} 
Finally,
\begin{align*}
g^{-1}t&=\beta^{-1}(1-t+\alpha t)t\\
&=\beta^{-1}\alpha t=\alpha t.
\end{align*}	
\end{proof}
We will now prove the promised gauge equivalence.
\begin{theorem}\label{decomposition} Let $g\in \hat{A}$ be defined as above. Then the MC elements $x$ and $\xi$ are gauge equivalent via $g$, i.e. $gxg^{-1}-d(g)g^{-1} = \xi$.
\end{theorem}	
\begin{proof}Let $k:=1+t-\alpha t$; we have seen that $g=k\beta$.
We have, taking into account (\ref{easy'}):
\begin{align*}\beta\cdot x &= \beta x\beta^{-1}-d(\beta)\beta^{-1}\\
&=\beta x\beta^{-1}-\beta x(1-t\beta)\beta^{-1}\\
&=\beta xt.
\end{align*}
Thus
\begin{align*}
	g\cdot x  & = k\cdot (\beta xt) \\
	&=  [k\beta xt  - d(k)]k^{-1} \\
	& = [(1+t-\alpha t)\beta xt+(\alpha t-1)x\alpha t](1-t+\alpha t) \\
	& = t\beta xt=\xi.
	\end{align*}
\end{proof}
\begin{rem}
	Using the anti-automorphism $\rho$, we can construct another invertible element  of $\hat{A}$ effecting gauge equivalence of $x$ and $\xi$. Indeed, since $\rho(x)=-x$ and $\rho(\xi)=-\xi$, it is easy to deduce from $g\cdot x=\xi$ that $\rho(g)^{-1}\cdot x = \xi$.  As a consequence the MC element $x$ is fixed under the gauge action of the $\rho$-invariant element $\rho(g)g$.
	
	Taking $\rho(g)^{-1}=(1-t+t\beta)(1+sx)=1-t+sx+t\beta$ in place of $g$, one  obtains slightly different versions of Corollary~\ref{conj}, Corollary~\ref{dgcattwist}, Theorem~\ref{decomposition_theorem} and Proposition~\ref{module} below.
	\end{rem}
\begin{cor}\label{conj}
Let $(V, d_V, s, t,x)$ be a dg $\hat{A}$-module. Then the following equality of operators on $V$ holds:
\[
(1+t-\alpha t)\beta (d_V+x) ((1+t-\alpha t)\beta)^{-1}=d_V+t\alpha xt.
\]
\end{cor}
\begin{proof}
The claimed equation is equivalent to \[(1+t-\alpha t)\beta x ((1+t-\alpha t)\beta)^{-1}-d_V(1+t-\alpha t)(1+t-\alpha t)^{-1}=d_V+t\alpha xt\]
which holds in $\End(V)$ because it holds in $\hat{A}$.
\end{proof}
We will now explain how this leads to an explicit form of the  decomposition theorem for $A_\infty$ algebras, \cite{Kajiura, lef, CL, LV}. The latter result is a strengthened version of the minimal model theorem; for an $A_\infty$ algebra $(V,d_V, m)$ with a choice of an HD, it gives an $A_\infty$ \emph{isomorphism} (as opposed to merely a weak equivalence) between $(V,d_V, m)$ and an $A_\infty$ algebra that is a direct sum of one supported on $tV$ and a linear contractible one supported on $(1-t)V$ (a linear contractible $A_\infty$ algebra has acyclic differential and vanishing higher products).

Let $(V,d,m)$ be an $A_\infty$ algebra with an HD as above. Recall that we have also operators $\tilde{s}$ and $\tilde{t}$ giving a multiplicative HD on $\hat{T}\Sigma V^*$. Consider the operator $m_{t(V)}=\tilde{t}m(1+\tilde{s}m)^{-1}\tilde{t}\in \End(\hat{T}\Sigma V^*)$.   It is equivalent to the original $m$ via the gauge transformation
$g=(1+\tilde{t}-(1+\tilde{s}m)^{-1}\tilde{t})(1+m\tilde{s})^{-1}\in\End(\hat{T}\Sigma V^*)$.
This is not yet the statement we are looking for. Firstly, $m_{t(V)}$  not a derivation of $\hat{T}\Sigma V^*$, but a $\tilde{t}$-derivation, i.e. for $a,b\in \hat{T}\Sigma V^*$ we have $m_{tV}(ab)=m_{tV}a\tilde{t}(b)+(-1)^{|a|}\tilde{t}(a)m_{tV}(b)$. Secondly,
$g$ is not a multiplicative automorphism of $\hat{T}\Sigma V^*$. Nevertheless it turns out, rather surprisingly, that a small modification to $g$ does produce the desired multiplicative isomorphism as we will now explain. 

For any (continuous, possibly non-multiplicative) endomorphism $h$ of $\hat{T}\Sigma V^*$, denote by $\overline{h}$ the \emph{multiplicative} endomorphism of $\hat{T}\Sigma V^*$ such that $\overline{h}|_{V^*}=h|_{V^*}$; this clearly determines $\overline{h}$ unambiguously. 
  It is easy to see that for any $v\in V^*$, we have $g^{-1}(v)=v+\text{terms  of tensor degree $>$1}$ from which it follows that $\overline{g^{-1}}$ is a multiplicative \emph{automorphism} of $\hat{T}\Sigma V^*$.
  As such, it acts by gauge transformations on derivations of $\hat{T\Sigma }V^*$ having degree quadratic or higher, i.e. $A_\infty$ structures on $V$. 

Furthermore, for an endomorphism (not necessarily a derivation) $f$ of $\hat{T}\Sigma V^*$ denote by $\check{f}$ the \emph{derivation} of $\hat{T}\Sigma V^*$ such that $\check{f}|_{V^*}=f|_{V^*}$.
Note that the derivation $\check{m}:=\check{m}_{tV}$ of $\hat{T}\Sigma V^*$ is an $A_\infty$ structure on $V$ that is a direct sum of a minimal $A_\infty$ structure on $V$ and a linear contractible one.

Then we have the following result, the promised decomposition theorem for $A_\infty$ algebras.

\begin{theorem}\label{decomposition_theorem}
The multiplicative automorphism $\overline{g^{-1}}$ determines a canonical (with respect to the HD on $V$) $A_\infty$ isomorphism $(V,m)\to(V,\check{m})$. More precisely, we have:
\begin{equation}\label{inverse} m=\overline{g^{-1}}\check{m}\overline{g^{-1}}^{-1}-d\overline{g^{-1}}~
	\overline{g^{-1}}^{-1}.
\end{equation}
\end{theorem}
\begin{proof}
The automorphism $g$ determines a gauge equivalence between $m_{tV}$ and $m$:
\[
m_{tV}=gmg^{-1}-dg g^{-1}.
\]
Evaluating this on an element $v\in V^*$ and multiplying through, we obtain:
\begin{equation}\label{inverse1}
g^{-1}m_{tV}(v)=mg^{-1}(v)+d(g^{-1})(v)
\end{equation}
Since $m_{tV}=\tilde{t}m_{tV}$ and $g^{-1}\tilde{t}=\alpha \tilde{t}$, which is a multiplicative automorphism of $\hat{T\Sigma V^*}$, and since $m_{tV}(v)=\check{m}(v)$, we can rewrite (\ref{inverse1}) as
\[
\overline{g^{-1}}\check{m}(v)=m\overline{g^{-1}}(v)+d(\overline{g^{-1}})(v).
\]
or, equivalently, as
\[
\check{m}(v)=\overline{g^{-1}}^{-1}m\overline{g^{-1}}(v)+\overline{g^{-1}}^{-1}d(\overline{g^{-1}})(v)
\]
Since the last equality holds for any $v\in V^*$, it follows that
\[
\check{m}=\overline{g^{-1}}^{-1}m\overline{g^{-1}}+\overline{g^{-1}}^{-1}d(\overline{g^{-1}})
\]
which is clearly equivalent to (\ref{inverse}).

\end{proof}		
\subsection{Decomposition theorem in categories of twisted objects}
Let $\C$ be a dg category, i.e. a category enriched in dg $\ground$-vector spaces.
\begin{defi} The dg category $\twC$ of twisted objects in $\C$ is defined as follows. Its objects are pairs $(M,x)$, where
$M\in\C$ and $x$ is a MC element in the dga $\End_{\C}(M)$, and the morphism spaces are given by $\Hom_{\twC}((M,x),(M',x'))=\Hom_\C(M,M')$ as graded spaces, with differential $d_{\twC}$ defined by
$d_{\twC}(f)=d_\C(f)+x'\circ f - (-1)^{|f|} f\circ x$.
\end{defi}
Note that the full subcategory of $\twC$ consisting of objects of the form $(M,0)$ is a equivalent to $\C$. Our definition differs to the standard one originally given in \cite{BK} in that the latter is obtained by first adding formally all direct sums to $\C$ and then forming the $\tw$-construction.
\begin{example}
Let $A$ be a dg algebra.
\begin{enumerate}
\item
Viewing $A$ as a dg category with one object, we see that $\tw(A)$ has MC elements in $A$ as objects. The dg space of morphisms between two MC elements $x$ and $y$ is $A^{[x,y]}$, whose underlying vector space is $A$ and the differential is $d^{[x,y]}$ so that for $a\in A$ we have $d^{[x,y]}(a)=d_A(a)+ya-(-1)^{a}ax$.  Note that $A^{[x,y]}$ is identified with the dg space of right $A$-module morphisms between $A^{[x]}$ and $A^{[y]}$.
\item
Take $\C$ to be the category of free $A$-modules. Then $\twC$ is what is normally called the category of twisted complexes of $A$-modules. Its objects are free $A$-modules with twisted differentials, and morphisms are the usual dg Hom-spaces of $A$-module morphisms.
\end{enumerate}
\end{example}

We have the following easy application of Theorem \ref{decomposition}.

\begin{cor}\label{dgcattwist}
Let $(M,x)$ be an object of $\twC$.  Suppose that $M$ is equipped with a Hodge decomposition, i.e, elements $t,s\in\End_{\C}(M)$ of degrees $0$ and $-1$, respectively, satisfying the
relations in Definition~\ref{def_H}, and assume that $1+xs$ (equivalently $1+sx$) is invertible.
Suppose further that the image of $t$ in $\C$ splits: there exists a direct sum decomposition $M = H \oplus K$ with $\End_{\C}(H)=t\End_\C(M)t$ and $\End_{\C}(K)=(1-t)\End_{\C}(M)(1-t)$. Then the automorphism
$(1+t-(1-sx)^{-1} t)(1-xs)^{-1}$ of $M$ induces an isomorphism $(M,x) \cong (H, tx(1+sx)^{-1}t) \oplus (K,0)$ in $\twC$, and therefore an isomorphism  $(M,x) \cong (H, tx(1+sx)^{-1}t)$ in $H^0(\twC)$.
\end{cor}
\begin{proof}
Because the element $1+xs$ is invertible as an endomorphism of $M$, the action of the dg algebra $A$ on $M$ extends to that of $\hat{A}$, i.e. there is a dg algebra map $\hat{A}\to \End_{\C}(M)$. Since $\xi=tx\alpha t\in \hat{A}$ is an MC element, its image $tx(1+sx)^{-1}t$ is an MC element in $\End_{\C}(M)$. Thus, $(M,x)$ and $(M,tx(1+sx)^{-1}t)$ are objects in $\twC$.

Furthermore, by Theorem \ref{decomposition} we have $\xi=gxg^{-1}-dg\cdot g^{-1}$
where $g=(1 + t - \alpha t)\beta$, which implies that $dg+\xi\cdot g-gx=0$. It follows that the element $g$, viewed as an endomorphism of $M$, determines a morphism $(M,tx(1+sx)^{-1}t)\to (M,x)$ in $\twC$, which is an isomorphism since $g$ is invertible. Finally, by assumption, there is an isomorphism $(M,tx(1+sx)^{-1}t)\cong (H, tx(1+sx)^{-1}t) \oplus (K,0)$ and we are done.
\end{proof}

We finish by applying our results to constructing an explicit minimal model for $A_\infty$ modules.
Let $(V,m)$ be an $A_\infty$ algebra and $(M, d^M)$ be a dg vector space. The structure of an $A_\infty$ module over $V$ on $M$ is a $\hat{T}\Sigma V^*$-linear differential $m^M$ on $\hat{T}\Sigma V^*\otimes M^*$ having the form $m^M=m^M_2+m^M_3+\ldots$ where $m^M_i:M^*\to \hat{T}^{i-1}\Sigma V^*\otimes M^*$. We will refer to a pair $(M,m^M)$ as an $A_\infty$ module over $(V,m)$. A map between $A_\infty$ modules $(M, m^M)$ and $(N,m^N)$ is a morphism of the corresponding $\hat{T}\Sigma V^*$ modules. Such a map will necessarily have the form $f=f_1+f_2+\ldots$ where $f_i:N^*\to \hat{T}^{i-1}\Sigma V^*\otimes M^*$; it is a weak equivalence if $f_1:N^*\to M^*$ is a quasi-isomorphism.  We refer to \cite{Ke} for a more detailed discussion of $A_\infty$ modules.

An $A_\infty$ module $(M,m^M)$ is \emph{minimal} if the differential $d^M$ vanishes. Similarly to the case of $A_\infty$ algebras, a weak equivalence between minimal $A_\infty$ modules is necessarily an isomorphism. On the opposite end of the spectrum we have a \emph{linear contractible} $A_\infty$ module: one having the form $(M,0)$ and with a contractible differential $d_M$.

Then we have the following decomposition theorem for $A_\infty$ modules as a direct consequence of Corollary \ref{dgcattwist}.
\begin{prop}\label{module}
Let $(M, m)$ be an $A_\infty$ module over an $A_\infty$ algebra $V$. Suppose that the dg vector space $(M,d^M)$ is supplied with operators $s,t$ turning it into an HD. Then the automorphism
$(1+t-(1-sm)^{-1} t)(1-ms)^{-1}$ of the dg space $(M, d_M)$ determines an isomorphism of $(M,m)$ to the direct sum of a minimal $A_\infty$ module $(tM, tm(1+sm)^{-1}t)$ and a linear contractible $A_\infty$ module $((1-t)M,d_M)$.
\end{prop}
\noproof

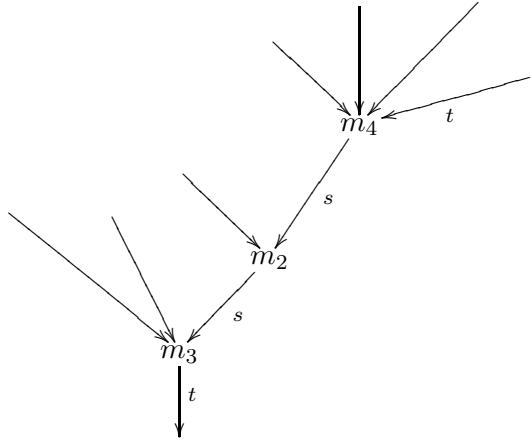
\begin{figure} [t]
	\[
	\xymatrix
	@C=3.5ex@R=2.50ex@M=0.3EX
	{
		&&&&& \ar[ddd] && \ar[dddll]\\
		&&&& \ar[ddr] \\
		&&&&&&&& \ar[dlll]^t\\
		&&&&& m_4 \ar[dddl]^s\\
		&&& \ar[drd] \\
		\ar[dddrrr] && \ar[dddr]\\
		&&&& m_2 \ar[ddl]^s &  \\
		& &&&&& \\
		&&& m_3 \ar[dd]^t \\
		\\
		&&&& \\
	}
	\]
	\caption{Definition of $m_\Gamma$.}
	\label{planartree}
\end{figure}

The (dual) structure maps $({m^{tM}_i})^*\colon T^{i-1}\Sigma^{-1} V \otimes tM \to tM$ for the minimal model $tM$ can be described explicitly as follows. Let $\Gamma$ be a planar rooted tree with $i+1$ extremities with the property that every branching has valence at least three and is below the rightmost leaf; such trees will be called \emph{admissible}. Label the root and the rightmost leaf by $t$, all other leaves by $\id$ and all internal edges by $s$. The following picture illustrates one such labelled tree. Define the map
$m_\Gamma \colon T^{i-1}\Sigma^{-1} V \otimes tM \to tM$ by composing labels in an obvious manner,
working from the leaves down to the root. For the tree pictured in Figure~\ref{planartree},
$$m_\Gamma = tm_3(\id^{\otimes 2}\otimes sm_2) (\id^{\otimes 3}\otimes sm_4)(\id^{\otimes 6} \otimes t).$$
We then put $({m^{tM}_i})^*=\sum m_\Gamma$, with the sum taken over all admissible trees. Note that no signs appear in this formula, since the action of the perturbation algebra $\hat{A}$ on $\hat{T}\Sigma V^*\otimes M^*$ dualises to
an action of its opposite dg algebra $\hat{A}^{\operatorname{op}}$ on $T\Sigma^{-1} V\otimes M$; the latter is then converted to an action of $\hat{A}$ via $\rho$.
Explicit formulae for the structure maps of the $A_\infty$-isomorphism
$g\colon (M,m)\cong (tM,tm(1+sm)^{-1}t)\oplus((1-t)M,d_M)$ and its inverse
can be obtained in a similar manner; we will refrain from giving details.

\end{document}